\documentclass[12pt]{amsart}

\usepackage{amsmath,amssymb,amsthm}

\usepackage{booktabs}
\usepackage[letterpaper,margin=1in,ignoreall]{geometry}
\usepackage{hyperref}
\usepackage[normalem]{ulem}
\usepackage{multirow}

\newtheorem{theorem}{Theorem}
\newtheorem{lemma}[theorem]{Lemma}

\newtheorem{proposition}[theorem]{Proposition}
\newtheorem{cor}[theorem]{Corollary}

\theoremstyle{remark}
\newtheorem{example}[theorem]{Example}
\newtheorem{remark}[theorem]{Remark}

\theoremstyle{definition}
\newtheorem{definition}[theorem]{Definition}

\usepackage{color}

\newcommand{\C}{\mathbb{C}}
\newcommand{\PP}{\mathbb{P}}

\newcommand{\rgen}{r_{\mathrm{gen}}}
\newcommand{\rmax}{r_{\mathrm{max}}}

\DeclareMathOperator{\Diff}{Diff}

\DeclareMathOperator{\Spec}{Spec}
\DeclareMathOperator{\length}{length}

\DeclareMathOperator{\al}{al}
\DeclareMathOperator{\algen}{al_{\mathrm{gen}}}

\newcommand{\defining}[1]{\textbf{#1}}

\newcommand{\smarthook}{\ensuremath{\mathbin{\text{\raisebox{.4ex}{%
  \vrule height .5pt width 1ex depth 0pt%
  \vrule height 0.8ex width .5pt depth 0pt%
}}\mathchoice{}{}{\mkern3mu}{\mkern3mu}}}}

\newcommand{\apolarityaction}{\smarthook}
\renewcommand{\aa}{\apolarityaction}

\newcommand{\veronese}{v}

\title{Some examples of forms of high rank}

\date{\today}

\author[J.~Buczy\'nski]{Jaros\l{}aw Buczy\'nski}
\thanks{J.~Buczy\'nski is supported by a grant Iuventus Plus of the Polish Ministry of Science,
 project 0301/IP3/2015/73, and by a scholarship of the Polish Ministry of Science.}
\address{Jaros\l{}aw Buczy\'nski\\
Faculty of Mathematics, Computer Science and Mechanics\\
University of Warsaw\\
ul. Banacha 2\\
02-097 Warszawa\\
Poland\\
and
Institute of Mathematics of the
  Polish Academy of Sciences\\
  ul. \'Sniadeckich 8\\
  00-656 Warszawa, Poland
}
 \email{jabu@mimuw.edu.pl}

\author[Z.~Teitler]{Zach Teitler}
\address{Zach Teitler \\
Boise State University \\
Department of Mathematics \\
1910 University Drive \\
Boise, ID 83725--1555}
\email{zteitler@boisestate.edu}

\subjclass[2010]{Primary: 13P05. Secondary: 15A69, 14N15}
\keywords{Waring rank}

\thanks{}
\begin{document}
\begin{abstract}
We describe some forms with greater Waring rank than previous examples.
In $3$ variables we give forms of odd degree with strictly greater rank than the ranks of monomials,
the previously highest known rank.
This narrows the possible range of values of the maximum Waring rank of forms in $3$ variables.
In $4$ variables we give forms of odd degree with strictly greater than generic rank.
In degrees greater than or equal to $5$ these are the first examples
showing that there exist forms with Waring rank strictly greater than the generic value.
\end{abstract}
\maketitle

\section{Introduction}

For a complex homogeneous form $F$ of degree $d$, the \defining{Waring rank} $r(F)$ is the least $r$ such that there exist
linear forms $\ell_1,\dotsc,\ell_r$ and scalars $c_1,\dotsc,c_r$ satisfying $F = c_1 \ell_1^d + \dotsb + c_r \ell_r^d$.
For example,
\[
  xyz = \frac{1}{24} \Big\{ (x+y+z)^3 - (x+y-z)^3 - (x-y+z)^3 - (-x+y+z)^3 \Big\}
\]
which shows $r(xyz) \leq 4$; and one can show in fact $r(xyz) = 4$.
For extensive introductions to Waring rank, including several different proofs that $r(xyz)=4$,
and including discussions of the history and applications of Waring rank,
see for example
\cite{MR1735271,MR2865915,Teitler:2014gf,MR2447451,Geramita,Reznick:2013uq}.

By the Alexander--Hirschowitz theorem \cite{MR1311347} a general form $F$ of degree $d > 1$ in $n$ variables
has rank $r(F)$ equal to
\[
  \left\lceil \frac{1}{n} \binom{n+d-1}{n-1} \right\rceil,
\]
except if $d=2$ (then $r(F) = n$) or $(n,d) = (3,4), (4,4), (5,4), (5,3)$ (then $r(F)$ is $1$ more than the above expression).
This value is called the \defining{generic rank}.
We denote it $\rgen(n,d)$.

It is an open question what is the maximum Waring rank of forms of degree $d$ in $n$ variables
for each $(n,d)$, known only in some small cases.
We write $\rmax(n,d)$ for the maximum Waring rank.
Of course the maximum rank must be greater than or equal to the rank of a general form: $\rmax(n,d) \geq \rgen(n,d)$.
Several upper bounds are known, such as $\rmax(n,d) \leq 2\rgen(n,d)$ \cite{MR3368091}
(see also \cite{MR2383331}, \cite{MR3196960}, \cite{Ballico:2013sf}).
For $d=2$ it is known that $\rmax(n,d) = \rgen(n,d) = n$.
For $n=2$, $d \geq 3$, it is known that $\rmax(n,d) = d > \rgen(n,d) = \lfloor \frac{d+2}{2} \rfloor$.
For larger values $n, d \geq 3$ much less is known.
One might ask whether the difference between the maximum Waring rank and the generic rank is unbounded.
But it is not even known whether this difference is positive, i.e., the maximum Waring rank is strictly
greater than the generic rank.
We focus on the latter question: for each $n,d \geq 3$ does there exist a form with rank strictly
greater than the generic rank?

The answer is known for some small cases.
For plane cubics $\rmax(3,3) = 5$ and $\rgen(3,3) = 4$,
see for example \cite[\textsection96]{MR0008171}, \cite{comonmour96}, \cite[\textsection8]{Landsberg:2009yq}.
For plane quartics $\rmax(3,4) = 7$ and $\rgen(3,4) = 6$,
see \cite{Kleppe:1999fk,Paris:2013fk}.
For cubic surfaces $\rmax(4,3) = 7$ while $\rgen(4,3) = 5$,
see \cite[\textsection97]{MR0008171}.
(See \cite{Holmes:2014hl} for the form $F = x_1 x_2^2 + x_3 x_4^2$ of degree $d=3$ in $n=4$ variables which has rank $6$.)
To our knowledge, the maximum Waring rank is not known up to now for any other values of $(n,d)$.

For $n=3$ and $d \geq 5$, while the maximum Waring rank is not yet known,
it is known that there exist forms with strictly greater than the generic rank.
The greatest Waring rank of a form in $3$ variables previously known is attained by monomials, see \cite{Carlini20125}.
Explicitly, if $d$ is odd, the monomial $x y^{(d-1)/2} z^{(d-1)/2}$ has rank
$r(x y^{(d-1)/2} z^{(d-1)/2}) = ((d+1)/2)^2$;
if $d$ is even, the monomial $x y^{(d-2)/2} z^{d/2}$ has rank $r(x y^{(d-2)/2} z^{d/2}) = d(d+2)/4$.
For $d \geq 5$ these are the greatest known ranks of forms in $3$ variables, until now.
In particular, for $d \geq 5$ their ranks are strictly greater than generic ranks.
See Table~\ref{table_ranks_in_3_variables}.

As far as we know,
these monomials in $3$ variables are the only forms in $n \geq 3$ variables known to have greater than the generic rank,
except in the cases $(n,d) = (3,3), (3,4), (4,3)$ discussed above, and one more example with $(n,d) =(5,3)$, see~\cite{MR3333949}.

We give a lower bound for Waring rank and some new examples of forms whose Waring ranks are strictly
greater than previously known examples.
\begin{table}[htb]
\textsc{Forms in $3$ variables}\\[0.25\baselineskip]
\begin{tabular}{ll rrr rrrrr rr}
\toprule
\multicolumn{2}{l}{degree} & 3 & 4 & 5  &  6 & 7 & 8 & 9 & 10  &  11 & 12 \\
\midrule
\multicolumn{2}{l}{Generic rank} & 4 & 6 & 7  &  10 & 12 & 15 & 19 & 22  &  26 & 31 \\
\hline
\multicolumn{2}{l}{Greatest rank of monomial} & 4 & 6 & 9  &  12 & 16 & 20 & 25 & 30  &  36 & 42 \\
\hline
\multirow{2}{*}{Maximum rank}&lower bound&\multirow{2}{*}{5} &\multirow{2}{*}{7}&\multirow{2}{*}{10}  & 12 & 17 & 20 & 26 & 30 & 37 & 42\\
                             &upper bound &  & &  &  19 & 24 & 30 & 40 & 44  &  60 & 62 \\
\bottomrule
\end{tabular}
\caption{Generic, maximum, and monomial ranks in $n=3$ variables.
The upper bound on maximum rank is provided by \cite{Ballico:2013sf,MR3368091,MR3349108}. 
The lower bound on maximum rank is mostly provided by monomials (even degrees $d \ge 6$), 
   and Theorem~\ref{thm: n=3 greater than monomial} (odd degrees $d\ge 5$).}\label{table_ranks_in_3_variables}
\end{table}

\begin{theorem}\label{thm: n=3 greater than monomial}
Let $d \geq 3$ be odd.
There exist forms of degree $d$ in $n = 3$ variables of rank strictly greater than $((d+1)/2)^2$,
the maximum rank of a monomial:
$\rmax(3,d) > ((d+1)/2)^2$.
\end{theorem}

In particular, De~Paris had previously shown that for forms of degree $d=5$ in $n=3$ variables the maximum Waring rank is either $9$ or $10$, see \cite{MR3349108}.
The monomial $xy^2z^2$ has $r(xy^2z^2) = 9$, and De Paris shows the upper bound $\rmax(3,5) \leq 10$.
We show that $\rmax(3,5) > 9$, i.e., there exists a form of rank $10$, so the maximum rank is $10$.
Explicitly we show that $F = xyz^3 + y^4z$ has $r(F) = 10$.

\begin{table}[htb]
\textsc{Forms in $4$ variables}\\[0.25\baselineskip]
\begin{tabular}{ll rrr rrrrr}
\toprule
\multicolumn{2}{l}{degree} & 3 & 4 & 5  &  6 & 7 & 8 & 9 & 10  \\
\midrule
\multicolumn{2}{l}{Generic rank} & 5 & 10 & 14  &  21 & 30 & 42 & 55 & 72 \\
\hline
\multicolumn{2}{l}{Greatest rank of monomial} & 4 & 8 & 12  &  18 & 27 & 36 & 48 & 64  \\
\hline
\multirow{2}{*}{Maximum rank}&lower bound& \multirow{2}{*}{7} &10& 15  &  21 & 31 & 42 & 56 & 72 \\
                             &upper bound& & 17 & 28 &  42 & 60 & 84 & 110 & 144 \\
\bottomrule
\end{tabular}
\caption{Generic, maximum, and monomial ranks in $n=4$ variables.
The upper bound on maximum rank is provided by \cite{Ballico:2013sf,MR3368091}. The lower bound on maximum rank is provided by generic rank (even degrees), and Theorem~\ref{thm: n=4 greater than generic} (odd degrees).}\label{table_ranks_in_4_variables}
\end{table}

And we show the following:
\begin{theorem}\label{thm: n=4 greater than generic}
Let $d \geq 3$ be odd.
There exist forms of degree $d$ in $n = 4$ variables of rank strictly greater than the generic rank:
$\rmax(4,d) > \rgen(4,d)$.
\end{theorem}
These are the first cases with $n \geq 4$, except for $(n,d) = (4,3)$ or $(5,3)$ mentioned previously.

The key idea for the lower bound that we use has been observed independently by Carlini, Catalisano, Chiantini, Geramita, and Woo,
and applied by them to show new cases of the Strassen Additivity Conjecture \cite{Carlini:2015sp}.
%

\section*{Acknowledgements}
We are grateful to Enrico Carlini, Luca Chiantini, and Alessandro De Paris for helpful comments.
The authors would like to thank Simons Institute for the Theory of Computing and the organizers
of the thematic semester ``Algorithms and Complexity in Algebraic Geometry''
for providing an excellent environment for scientific activity.
The article is written as a part of ``Computational complexity, generalised Waring type problems and tensor decompositions'',
a project within ``Canaletto'',
    the executive program for scientific and technological cooperation between Italy and Poland, 2013-2015.
The paper is also a part of the activities of the AGATES research group.

\section{Preliminaries}

We work over the complex numbers $\C$.
Fix $S = \C[x_1,\dotsc,x_n]$ and the dual ring $T = \C[\alpha_1,\dotsc,\alpha_n]$
acting on $S$ by letting each $\alpha_i$ act as $\partial/\partial x_i$; this is called the \defining{apolarity} action.
We denote it by the symbol $\aa$, as in $\alpha_i \aa x_i^k = k x_i^{k-1}$.
In small dimensions we may take variables $x, y, z$ and dual variables $\alpha, \beta, \gamma$.
In any case, elements of $S$ are denoted by Roman letters and elements of $T$ are denoted by Greek letters.
For example, $\alpha^2 \beta^3 \aa x^4 y^5 = 240 x^2 y^3$.

For a vector space $V$ we denote by $\PP V$ the projective space of lines in $V$.
For a nonzero vector $v \in V$ we write $[v] \in \PP V$ for the line in $V$ spanned by $v$.
For an ideal $I\subset T$ or form $\Theta \in T$ we write $V(I)$ or $V(\Theta)$ for the affine scheme or variety in $S_1 = T_1^*$ defined by $I$ or $\Theta$.
When $I$ or $\Theta$ is homogeneous we write $\PP V(I)$ or $\PP V(\Theta)$ for the corresponding projective scheme or variety.

For $F \in S$ let $F^\perp = \{ \Theta \in T \mid \Theta \aa F = 0 \}$, the \defining{apolar} or \defining{annihilating ideal} of $F$.
Recall the Apolarity Lemma, that for a scheme $Z \subset \PP^{n-1} = \PP S_1$ with saturated homogeneous ideal $I$,
$[F]$ lies in the linear span of the Veronese image $\veronese_d(Z)$ if and only if $I \subset F^\perp$;
see for example \cite[Lemma 1.15]{MR1735271}.
When $Z = \{[\ell_1],\dotsc,[\ell_r]\}$ is reduced, this says there are scalars $c_i$ such that
$F = \sum c_i \ell_i^d$ if and only if $I \subset F^\perp$.
We may replace each $\ell_i$ by $c_i^{1/d} \ell_i$ and write simply $F = \sum \ell_i^d$;
so $F = \sum \ell_i^d$, up to scaling, if and only if $I = I(\{[\ell_1],\dotsc,[\ell_r]\}) \subset F^\perp$.
Hence the Waring rank $r(F)$ is the least length of a reduced saturated homogeneous one-dimensional ideal $I \subset F^\perp$.
A scheme $Z$ or ideal $I$ is called \defining{apolar to $F$} if $I \subset F^\perp$,
equivalently if $[F]$ lies in the span of the $d$'th Veronese image of $Z$;
so the Waring rank of $F$ is equal to the least length of a zero-dimensional reduced apolar scheme to $F$.
A typical approach to giving lower bounds for $r(F)$ is to analyze reduced apolar schemes to $F$.
This is the approach we take here.

Some related notions are worth mentioning.
The \defining{cactus rank} $cr(F)$, or \defining{scheme length},
of $F$ is the least length of a saturated homogeneous one-dimensional ideal $I \subset F^\perp$ (not necessarily reduced).
The \defining{smoothable rank} $sr(F)$ is the least length of a smoothable zero-dimensional apolar scheme
(recall that a scheme is smoothable if it lies in an irreducible family whose general member is smooth).
The $r$'th \defining{secant variety} of the Veronese variety is the Zariski closure of the locus of forms of Waring rank $r$.
The \defining{border rank} $br(F)$ is the least $r$ such that $[F]$ lies in the $r$'th secant variety,
that is, $F$ is a limit of forms of rank $r$.
Evidently $cr(F) \leq sr(F) \leq r(F)$ and $br(F) \leq r(F)$.
In fact $br(F) \leq sr(F)$.
All these inequalities may be strict, or may be equalities.
For examples with $cr(F) < br(F)$, see for instance \cite{MR2996880}.
For examples with $cr(F) > br(F)$, see \cite{MR3333949}.

Let $A^F = T/F^\perp$, the \defining{apolar algebra} of $F$.
Let $\Diff(F) = T \aa F = \{ \Theta \aa F \mid \Theta \in T \} \subset S$;
note $\Diff(F) \cong A^F$ as $\C$-vector spaces.
Recall (see for example \cite{Derksen:2014hb}) that for any $F \in S$, $\Theta \in T$
we have
\begin{equation}\label{equ: colon ideal and derivation}
         F^\perp : \Theta = (\Theta \aa F)^\perp
      \end{equation}
and we have the short exact sequence
\[
  0 \to T/(F^\perp : \Theta) \overset{\Theta}{\longrightarrow} T/F^\perp \to T/(F^\perp + \Theta) \to 0.
\]
In particular $\length(T/(F^\perp + \Theta)) = \dim \Diff(F) - \dim \Diff(\Theta \aa F)$.
Let $\al(F) = \length(A^F) = \dim \Diff(F)$, the \defining{apolar length} of $F$,
so
\[
  \length(T/(F^\perp + \Theta)) = \al(F) - \al(\Theta \aa F) = \length(A^F/\Theta A^F) .
\]

The following was essentially observed in \cite{Derksen:2014hb}.
\begin{proposition}\label{prop: no point bound}
Let $F \in S$ be a homogeneous form of degree $d$,
let $\alpha \in T_1$ be a linear form,
and let $I \subseteq F^\perp$ be a saturated homogeneous one-dimensional apolar ideal.
Suppose that the zero-dimensional scheme $\PP V(I)$ has no point of support on the hyperplane $\PP V(\alpha)$;
equivalently, $I = I:\alpha$.
Then $\deg I \geq \al(F) - \al(\alpha \aa F)$.
\end{proposition}
\begin{proof}
From $I + \alpha \subseteq F^\perp + \alpha$ we get
$\Spec(T/(F^\perp + \alpha)) \subseteq V(I) \cap V(\alpha)$, a proper intersection by hypothesis,
having length equal to $\deg I$.
Thus $\deg I \geq \length(T/(F^\perp + \alpha)) = \al(F) - \al(\alpha \aa F)$.
\end{proof}
This does not require $I$ to be reduced, so it leads to a bound for cactus rank $cr(F)$.
The hypothesis that $\PP V(I)$ has no point of support on $\PP V(\alpha)$ can be realized by,
for example, taking $\alpha$ general:
for $\alpha$ general, $cr(F) \geq \al(F) - \al(\alpha \aa F)$;
this is Theorem 3.1 of \cite{Derksen:2014hb}.

Here are the new observations which form the starting point for this paper.

\begin{proposition}\label{prop: bound points off alpha}
Let $F \in S$ be a homogeneous form of degree $d$,
let $\alpha \in T_1$ be a linear form,
and let $I \subseteq F^\perp$ be a reduced saturated homogeneous one-dimensional apolar ideal.
Then $\deg(I:\alpha) \geq \al(\alpha \aa F) - \al(\alpha^2 \aa F)$.
In particular $Z = \PP V(I)$ has at least $\al(\alpha \aa F) - \al(\alpha^2 \aa F)$ points of support off of the hyperplane $\PP V(\alpha)$.
\end{proposition}
\begin{proof}
Note $I:\alpha$ is a saturated homogeneous ideal and $I:\alpha \subset F^\perp:\alpha = (\alpha \aa F)^\perp$.
And $\PP V(I:\alpha)$ has no point of support on $\PP V(\alpha)$.
The result follows by Proposition~\ref{prop: no point bound}.
\end{proof}

\begin{remark}
If $Z$ is a zero-dimensional scheme with multiplicity at most $k$ at each support point in $\PP V(\alpha)$
then $Z - (Z \cap \PP V(\alpha))$ has length at least $\al(\alpha^k \aa F) - \al(\alpha^{k+1} \aa F)$.
\end{remark}

\begin{remark}
In particular this ignores multiplicities (or reducedness) of $Z$ outside of $\PP V(\alpha)$.
At this time we do not know how to exploit reducedness of $Z$ outside of $\PP V(\alpha)$
to give an improved bound.
\end{remark}

A somewhat more general version of the next statement was observed independently
by Carlini, Catalisano, Chiantini, Geramita, and Woo \cite{Carlini:2015sp}.
\begin{theorem}\label{thm: waring bound}
Let $F \in S$ be a homogeneous form of degree $d$
and let $\alpha \in T_1$ be a linear form.
Then $r(F) \geq \al(\alpha \aa F) - \al(\alpha^2 \aa F)$.
\end{theorem}
\begin{proof}
Let $I \subset F^\perp$ be a reduced apolar ideal of degree $\deg I = r(F)$.
Then $r(F) = \deg I \geq \deg(I:\alpha) \geq \al(\alpha \aa F) - \al(\alpha^2 \aa F)$
by Proposition~\ref{prop: bound points off alpha}.
\end{proof}

The next statement does not seem to have been previously observed, to our knowledge.
\begin{cor}\label{corollary: improve by 1}
If $\al(F) - \al(\alpha \aa F) > \al(\alpha \aa F) - \al(\alpha^2 \aa F)$ then $r(F) > \al(\alpha \aa F) - \al(\alpha^2 \aa F)$.
\end{cor}
\begin{proof}
Let $I \subset F^\perp$ be a reduced apolar ideal of degree $\deg I = r(F)$.
By Proposition~\ref{prop: bound points off alpha} $\PP V(I)$ has at least $\al(\alpha \aa F) - \al(\alpha^2 \aa F)$ points off of $\PP V(\alpha)$.
If $r(F) = \al(\alpha \aa F) - \al(\alpha^2 \aa F) = \deg(I)$ then $\PP V(I)$ has no support on $\PP V(\alpha)$.
In this case Proposition~\ref{prop: no point bound} yields $r(F) = \deg I \geq \al(F) - \al(\alpha \aa F)$, as claimed.
\end{proof}

\begin{example}\label{example: rank of special form}
Let $F = G(x)H(y) + K(y)$, where $x$ and $y$ denote tuples of independent variables,
and suppose $\alpha \in T_1$ is differentiation by one of the $x$ variables, so that $\alpha \aa H = \alpha \aa K = 0$.
Then $\alpha \aa F = (\alpha \aa G)H$ and $\Diff(\alpha \aa F) \cong \Diff(\alpha \aa G) \otimes \Diff(H)$.
Here $\otimes$ denotes the usual tensor product of complex vector spaces, 
 and the isomorphism follows since $\alpha \aa G$ and $H$ are polynomials in independent variables.
Similarly $\alpha^2 \aa F = (\alpha^2 \aa G)H$ and $\Diff(\alpha^2 \aa F) \cong \Diff(\alpha^2 \aa G) \otimes \Diff(H)$.
Then
\[
  r(F) \geq (\dim \Diff(\alpha \aa G) - \dim \Diff(\alpha^2 \aa G))(\dim \Diff(H)) = (\al(\alpha \aa G) - \al(\alpha^2 \aa G)) \al(H) .
\]

In particular $r(x^a H(y) + K(y)) \geq \al(H)$.

Especially, let $F = x_1^{a_1} \dotsm x_n^{a_n}$, $1 \leq a_1 \leq \dotsb \leq a_n$, $\alpha = \alpha_1$.
Then we obtain $r(F) \geq \al(x_2^{a_2} \dotsm x_n^{a_n}) = (a_2+1) \dotsm (a_n+1)$.
This recovers the theorem of Carlini--Catalisano--Geramita on Waring ranks of monomials \cite{Carlini20125},
see also \cite{MR2842085,MR3017012}.
(In fact the proof given by Carlini--Catalisano--Geramita is quite close to the idea of Theorem~\ref{thm: waring bound}.)
\end{example}

\begin{example}\label{example: complete intersection}
It is shown in \cite{MR2842085} that
if $F^\perp$ is a complete intersection generated in degrees $d_1 \leq \dotsb \leq d_n$
then $cr(F) = d_1 \dotsm d_{n-1} \leq r(F) \leq d_2 \dotsm d_n$.

Suppose $F^\perp = (\phi_1,\dotsc,\phi_n)$ is a complete intersection with
$\deg \phi_i = d_i$ for each $i$, where $d_1 \leq \dotsb \leq d_n$,
and suppose $\alpha \in T_1$ is such that $\alpha^2 \mid \phi_1$.
Note that $F^\perp : \alpha = (\phi_1/\alpha,\phi_2,\dotsc,\phi_n)$
and $F^\perp : \alpha^2 = (\phi_1/\alpha^2,\phi_2,\dotsc,\phi_n)$.
Hence $r(F) \geq \al(\alpha \aa F) - \al(\alpha^2 \aa F) = d_2 \dotsm d_n \geq r(F)$.

This generalizes the example of monomials.
Compare Theorem~4.14 of \cite{Carlini:2015sp}.
\end{example}

\section{Forms with higher than general rank}

We adopt a slightly modified form of notation of \cite{MR1735271}:
\begin{definition}
Fix integers $n,d,s$. Recall that $T = \C[\alpha_1,\dotsc,\alpha_n]$.
Let $H(n,d)$ be the function $H(n,d)(i) = \min\{\dim_\C T_i, \dim_\C T_{d-i}\}$.
Let $H(n,d,s)$ be the function $H(n,d,s)(i) = \min\{\dim_\C T_i, \dim_\C T_{d-i}, s\}$.
\end{definition}
(In \cite{MR1735271} these are written $H(d,n)$ and $H(s,d,n)$ respectively,
although \cite{MR1735271} uses $j$ in place of $d$ and $r$ in place of $n$.)
As usual we may write these functions by writing their sequences of values for $i = 0, 1, \dotsc$:
thus, for example, $H(3,6,8) = 1,3,6,8,6,3,1$, all subsequent values being zero.

Recall the following well-known facts.
\begin{proposition}\label{prop_Hilbert_function_of_general_form}
The Hilbert functions of apolar algebras behave as follows.
\begin{enumerate}
\item \textup{(}\cite[Prop.~3.12]{MR1735271}\textup{)}
Fix integers $n$ and $d$. Let $G \in S_d$ be general.
Then the Hilbert function of $A^G$ is $H(n,d)$.
\item \textup{(}\cite[Lemma 1.17]{MR1735271}\textup{)}\label{item_Hilb_function_of_general_sum_of_powers}
Fix integers $n, d, s$. Let $\ell_1,\dotsc,\ell_s \in S_1$ be general linear forms
and $G = \ell_1^d + \dotsb + \ell_s^d$.
Then the Hilbert function of $A^G$ is $H(n,d,s)$.
\end{enumerate}
\end{proposition}
In the first case the algebra $A^G$ is called \defining{compressed}
(see \cite{MR1735271} for a more general notion of compressed algebras which are not necessarily Gorenstein or graded).
These statements hold also in positive characteristic by taking $G$ to be a DP-form, see \cite{MR1735271}.

\begin{lemma}\label{lem_apolar_length_of_general_form}
\begin{enumerate}
\item \label{item_apolar_length_of_general_form_in_S_d}
Fix integers $n$ and $d$. Let $G \in S_d$ be any form such that $A^G$ has Hilbert function $H(n,d)$.
Then the apolar length of $G$ is $\al(G) = \binom{n + \lfloor (d-1)/2 \rfloor}{n} + \binom{n + \lceil (d-1)/2 \rceil}{n}$.

\item \label{item_apolar_length_of_general_form_in_sigma_s}
Fix integers $n,d,s$.
Let $G \in S_d$ be any form such that $A^G$ has Hilbert function $H(n,d,s)$.
Suppose $\dim T_i \leq s < \dim T_{i+1}$, where $i < d/2$.
Then the apolar length of $G$ is $\al(G) = 2\binom{n+i}{i} + s(d-2i-1)$.
\end{enumerate}
\end{lemma}

The proof is an easy computation which we leave to the reader.

We write $\algen(n,d)$ for the apolar length of a general form in $n$ variables of degree $d$;
that is, $\algen(n,d) = \binom{n+\lfloor (d-1)/2 \rfloor}{n} + \binom{n + \lceil (d-1)/2 \rceil}{n}$.

Before we produce forms with strictly greater rank than previously known examples,
we carry out some preliminary computations that involve producing new forms
with rank at least as great as previously known examples.

By Theorem~\ref{thm: waring bound} (or Example~\ref{example: rank of special form}),
$r(x_1 H(x_2,\dotsc,x_n)+K(x_2,\dotsc,x_n)) \geq \al(H)$,
independent of the choice of $K$.
In particular if $H$ is general this shows that
\begin{equation}\label{eqn: bound with H general}
  \rmax(n,d) \geq \algen(n-1,d-1).
\end{equation}
An easy computation by hand shows for $d$ odd,
\begin{equation}\label{eqn: general apolar length n=3}
  \algen(3,d-1) = \rgen(4,d)
\end{equation}
(the left hand side is a binomial formula in 
       Lemma~\ref{lem_apolar_length_of_general_form}.(\ref{item_apolar_length_of_general_form_in_S_d});
       the right hand side is given by the Alexander-Hirschowitz Theorem).
It is also easy to see that $\algen(3,d-1) < \rgen(4,d)$ for $d$ even;
$\algen(n-1,d-1) < \rgen(n,d)$ for $n \geq 5$ and $d \gg 0$;
on the other hand $\algen(2,d-1) > \rgen(3,d)$ for $d \geq 5$.

\begin{example}
Let $H(y,z)$ be a general binary form of degree $d-1$ and let $K(y,z)$ be an arbitrary binary form of degree $d$.
Then $r(xH + K) \geq \al(H)$.
Since $H$ is general we compute $\al(H) = (d^2+2d)/4$ if $d$ is even,
$(d+1)^2/4$ if $d$ is odd.
In any case $\al(H) \approx d^2/4$.
By the Alexander--Hirschowitz theorem
the general rank of a form of degree $d$ in $3$ variables is
\[
  \left\lceil \frac{1}{3} \binom{d+2}{2} \right\rceil = \left\lceil \frac{(d+2)(d+1)}{6} \right\rceil \approx \frac{d^2}{6},
\]
or one more than this if $d=4$.
Thus the forms $xH+K$ have higher than general rank for $d$ large enough; $d \geq 5$ will do.
Note that this is independent of the choice of $K$!
The ternary monomials considered in \cite{Carlini20125}
are given by $H = y^{\lfloor (d-1)/2 \rfloor} z^{\lceil (d-1)/2 \rceil}$, $K = 0$.
\end{example}

\begin{example}
In $n=4$ variables, with $d \geq 3$ odd,
for $H(x_2,x_3,x_4)$ general of degree $d-1$ and $K(x_2,x_3,x_4)$ arbitrary of degree $d$,
$F = x_1 H + K$ has rank $r(F) \geq \al(H) = \algen(3,d-1) = \rgen(4,d)$.
So these forms have rank at least as great as general rank.
\end{example}

So taking $H$ general, and $K$ arbitrary, shows explicitly that $F$ realizes the obvious inequality
$\rmax(4,d) \geq \rgen(4,d)$ for $d$ odd;
and $\rmax(3,d)$ is greater than or equal to the maximum rank of a ternary monomial.
Now the idea is that we can improve \eqref{eqn: bound with H general} by choosing $H$ to be not general,
and $K$ meeting certain conditions.

\begin{lemma}\label{lemma: surjectivity of differentiation}
For any $a \geq 0$ and any nonzero $\Psi \in T_b$
the linear map $D = D_{\Psi,a} : S_{a+b} \to S_a$, $F \mapsto \Psi \aa F$,
is surjective.
\end{lemma}
\begin{proof}
Fix a monomial order $<$, such as lexicographic order.
Let $\alpha^{m_0}$ be the $<$-last monomial in $\Psi$.
For any monomial $x^m$ of degree $a$, $x^m$ is the leading monomial of $\Psi \aa x^{m+m_0}$.
This gives a triangular system of linear equations whose solution expresses each monomial $x^m$
as an element of the image of $D_{\Psi,a}$.
\end{proof}

\begin{theorem}\label{thm: general case higher rank}
Let $n \geq 3$ and $d = 2k+1 \geq 3$.
There exists a form of degree $d$ in $n$ variables of rank strictly greater than the
apolar length of a general form of degree $d-1$ in $n-1$ variables:
$\rmax(n,d) > \algen(n-1,d-1)$.
\end{theorem}
\begin{proof}
We use the $n$ variables $x_1, x_2,\dotsc,x_n$; for convenience we write $x = x_1$ and $\alpha = \alpha_1$.
Let $s = \binom{n+k-2}{k} - 1$ and let $G(x_2,\dotsc,x_n)$ be a general sum of $s$ $(d-1)$st powers of linear forms
in variables $x_2,\dotsc,x_n$.
Let $F = xG + K(x_2,\dotsc,x_n)$, with $K$ a form of degree $d$ to be determined later.
Eventually, $K$ will be a general form of degree $d$ in $n-1$ variables, however, for the sake of argument, we do not assume anything on $K$ yet.
Since $\alpha \aa F = G$ and $\alpha^2 \aa F = 0$ we get $r(F) \geq \al(G)$.
By construction and
Proposition~\ref{prop_Hilbert_function_of_general_form}~(\ref{item_Hilb_function_of_general_sum_of_powers})
$A^G$ has Hilbert function $H(n-1,d-1,s)$
and $\al(G) = \algen(n-1,d-1)-1$.
That is,
\[
\begin{split}
  r(F) &\geq \al(G) = \binom{n-1 + \lfloor (d-2)/2 \rfloor}{n-1} + \binom{n-1 + \lceil (d-2)/2 \rceil}{n-1} - 1 \\
    &= \binom{n+k-2}{n-1} + \binom{n+k-1}{n-1} - 1.
\end{split}
\]
This holds regardless of the choice of $K$.

We have $G^\perp = F^\perp : \alpha$ by \eqref{equ: colon ideal and derivation},
  so $F^\perp \subseteq G^\perp$.
From the Hilbert function of $A^G$ we see that $G^\perp$ has the minimal generator $\alpha$,
a single minimal generator in degree $k$, and all other minimal generators must be in degrees $k+1$ or higher.
It follows that for degrees $2 \leq i \leq k-1$ we have $(F^\perp)_i \subseteq (G^\perp)_i = (\alpha)_i$.
But if $\alpha \Theta \in F^\perp$ for some $\Theta \in T_{i-1}$ then $\Theta \in (G^\perp)_{i-1} = (\alpha)_{i-1}$,
so $\alpha \Theta \in (\alpha^2)$.
This shows $(F^\perp)_i = (\alpha^2)_i$ for $2 \leq i \leq k-1$.

Now, let $K$ be chosen so that $(F^\perp)_k = (\alpha^2)_k$.
We will show later that there exists an open dense subset of such $K$,
in fact satisfying an additional constraint that we will describe.

From this we can compute the apolar length of $F$:
\[
\begin{split}
  \al(F) &= 2\left\{ 1 + n + \left( \tbinom{n+1}{2} - 1 \right) + \left( \tbinom{n+2}{3} - n \right) + \dotsb
    + \left( \tbinom{n+k-1}{k} - \tbinom{n+k-3}{k-2} \right) \right\} \\
    &= 2\left\{ \binom{n+k-2}{k-1} + \binom{n+k-1}{k} \right\} \\
    &= 2 \al(G) + 2 \\
    &= 2 \algen(n-1,d-1).
\end{split}
\]
By Corollary~\ref{corollary: improve by 1} we get
\[
  r(F) \geq \binom{n+k-2}{n-1} + \binom{n+k-1}{n-1} = \al(G) + 1 = \algen(n-1,d-1).
\]
So far, this is the same value we would get by taking $G$ to be general.
Now we will show that we can increase the bound on $r(F)$ by $1$.

We claim that $r(F) \geq \al(G) + 2$.
So, suppose to the contrary that $r(F) = r = \binom{n+k-2}{n-1} + \binom{n+k-1}{n-1} = \al(G)+1$.
Let $F = \ell_1^d + \dotsb + \ell_r^d$ and let $I = I(\{[\ell_1],\dotsc,[\ell_r]\})$.
By Proposition~\ref{prop: bound points off alpha}, there must be at least $\al(G)$ points off of the hyperplane $\PP V(\alpha)$.
If all of the points $[\ell_i]$ are off of $\PP V(\alpha)$ then by Proposition~\ref{prop: no point bound}
we have in fact $r(F) \geq \al(F) - \al(\alpha \aa F) = \al(F) - \al(G) = \al(G)+2$, giving the claimed improvement.
(Here the fact $\al(G)$ is $1$ less than generic means $\al(F) - \al(G)$ is $1$ more than we would have if $G$ were generic;
this is where the non-genericity of $G$ gives an improvement in the bound for $r(F)$.)
Otherwise there is exactly one $[\ell_i]$ lying on $\PP V(\alpha)$.
Without loss of generality $[\ell_r]$ lies on $\PP V(\alpha)$ and the others lie off of it.
That is, $\ell_r = \ell_r(x_2,\dotsc,x_n)$ does not depend on $x$.
Let $F' = F - \ell_r^d = \ell_1^d + \dotsb + \ell_{r-1}^d = xG + (K - \ell_r^d)$.
We will choose $K$ in such a way that $(F^\perp)_k = (F'^\perp)_k = (\alpha^2)_k$.
Then the above arguments will apply to $F'$ and give us $r(F') \geq \al(G)+1$.
That is, $r-1 \geq r(F') \geq \al(G) + 1$.
Thus $r(F) \geq \al(G) + 2$, as claimed.

What is left is to show that there exists some $K$ such that $(F^\perp)_k = (\alpha^2)_k$
and for any linear form $\ell = \ell(x_2,\dotsc,x_n)$, $((F - \ell^d)^\perp)_k = (\alpha^2)_k$.

Let $\Psi \in (G^\perp)_k$ be the minimal generator of $G^\perp$ of degree $k$.
Since $\alpha \in G^\perp$ we can take $\Psi$ to only involve $\alpha_2,\dotsc,\alpha_n$.
Recall that $T_{k-1} \aa G \subseteq S_{k+1}$ is the subspace consisting of $(k-1)$st derivatives of $G$;
we have $\dim T_{k-1} \aa G = \binom{n+k-3}{n-2}$ by the Hilbert function of $A^G$.
Let $S' \subset S$ be the subring $\C[x_2,\dotsc,x_n]$. 
Since $G \in S'$, its derivatives also do not involve $x_1$, 
   in particular $T_{k-1} \aa G \subseteq S'_{k+1}$.
But $\dim S'_{k+1} = \binom{n+k-1}{n-2}$.
So $T_{k-1} \aa G \subsetneqq S'_{k+1}$.

Let $K \in \C[x_2,\dotsc,x_n]_d$ be any form so that $\Psi \aa K \notin T_{k-1} \aa G$.
There exist a plethora of such forms by Lemma~\ref{lemma: surjectivity of differentiation}.

With such a choice we claim $(F^\perp)_k = (\alpha^2)_k$.
Suppose $\Theta = \Theta(\alpha,\alpha_2,\dotsc,\alpha_n) \in (F^\perp)_k$.
We may discard all terms containing $\alpha^2$, so we may write $\Theta = \alpha \phi + \psi$
where $\phi, \psi$ only involve $\alpha_2,\dotsc,\alpha_n$, and $\phi \in T_{k-1}$, $\psi \in T_k$.
Then $0 = \Theta F = x (\psi \aa G) + \phi \aa G + \psi \aa K$, so $\psi \aa G = \phi \aa G + \psi \aa K = 0$.
Thus $\psi \in (G^\perp)_k$.
Since $\psi$ only involves $\alpha_2,\dotsc,\alpha_n$, $\psi = c \Psi$ for some $c \in \C$.
So $0 = \phi \aa  G + \psi\aa  K = \phi \aa G + c \Psi \aa K$.
Since $\Psi \aa K \notin T_{k-1} \aa G$ it must be $c = 0$ and $\phi \aa G = 0$, so $\Theta = \alpha \phi$
where $\phi \in (G^\perp)_{k-1} = (\alpha)_{k-1}$.
Then $\Theta \in (\alpha^2)_k$.

Now the idea is to choose $K \in \C[x_2,\dotsc,x_n]_d$ a form so that not only $\Psi\aa  K \notin T_{k-1}\aa  G$,
but in fact $\Psi\aa (K - \ell^d) \notin T_{k-1}\aa  G$ for all $\ell = \ell(x_2,\dotsc,x_n)$.
The linear map 
  $D = D_{\Psi,k+1} \colon S'_d \to S'_{k+1}$ is surjective,
so $D^{-1}(T_{k-1} \aa G)$ has codimension equal to the codimension of $T_{k-1} \aa G$, which is $\binom{n+k-1}{n-2} - \binom{n+k-3}{n-2}$.
The projective Veronese variety $\{[\ell^d] : \ell = \ell(x_2,\dotsc,x_n)\}$ has dimension $n-2$.
We have $\binom{n+k-1}{n-2} - \binom{n+k-3}{n-2} = \binom{n+k-3}{n-3} + \binom{n+k-2}{n-3} \geq \binom{n-2}{n-3}+\binom{n-1}{n-3} > n-2$.
So a general translate of the Veronese variety is disjoint from $\PP(D^{-1}(T_{k-1} \aa G))$.
This shows that for general $K$, $\Psi\aa (K-\ell^d) \notin T_{k-1} \aa G$ as claimed.

By the above calculation, $((F-\ell^d)^\perp)_k = (\alpha^2)_k$ so $r(F-\ell^d) \geq \al(G)+1$ for all $\ell = \ell(x_2,\dotsc,x_n)$.
As discussed above, then $r(F) \geq \al(G)+2$.
\end{proof}

\begin{proof}[Proof of Theorem~\ref{thm: n=3 greater than monomial}]
The apolar length of a general binary form of degree $d-1=2k$ is $(k+1)^2 = ((d+1)/2)^2$.
So there exists a form of degree $d$ in $3$ variables of rank strictly greater than $((d+1)/2)^2$, as claimed.
\end{proof}

\begin{proof}[Proof of Theorem~\ref{thm: n=4 greater than generic}]
There exists a form in $4$ variables of degree $d$ of rank strictly greater than
the apolar length of a general form in $3$ variables of degree $d-1=2k$,
which is $\binom{k+2}{3} + \binom{k+3}{3}$,
which is equal to the generic rank of a form in $4$ variables of degree $d$.
\end{proof}

The genericity conditions in the proof of Theorem \ref{thm: general case higher rank} are very explicit and can be easily applied in practice.
We illustrate this in the case of ternary quintics.

De Paris has shown that every ternary quintic (form of degree $d=5$ in $n=3$ variables)
has Waring rank at most $10$, see \cite{MR3349108}.
It is well known $r(x y^2 z^2) = 9$.
But it is left open by De Paris whether the maximum rank of a ternary quintic is $9$ or $10$.

\begin{theorem}
There exists a ternary quintic form of rank $10$.
Explicitly, $F = xyz^3 + y^4z$ has $r(F) = 10$.
\end{theorem}
\begin{proof}
Here is an explicit expression showing $r(F) \leq 10$:
$F = (xyz^3 - 2y^2z^3 - (1/5)z^5) + (y^4z + 2y^2z^3 + (1/5)z^5)$.
Here $(y^4z + 2y^2z^3 + (1/5)z^5)^\perp = (\beta^2 - \gamma^2 , \beta \gamma^4)$,
and since $\beta^2-\gamma^2$ has distinct roots, this binary form has $r(y^4z + 2y^2z^3 + (1/5)z^5) = 2$.
And compute $(xyz^3 - 2y^2z^3 - (1/5)z^5)^\perp =  (\alpha^2 , 4 \alpha \beta + \beta^2 , \beta^2 \gamma^2 - \gamma^4 )$,
which is a complete intersection generated in degrees $2, 2, 4$
with the first generator divisible by (equal to) the square of a linear form;
by Example~\ref{example: complete intersection} $r(xyz^3 - 2y^2z^3 - (1/5)z^5) = 2 \cdot 4 = 8$.
Thus $r(F) \leq 2 + 8 = 10$.

However the more important point is to show $r(F) \geq 10$.
We compute:
\begin{align*}
  \alpha^2 \aa F &= 0, & \beta^2 \aa F &= 12 y^2 z, \\
  \alpha \beta \aa F &= z^3, & \beta \gamma \aa  F &= 3xz^2 + 4y^3, \\
  \alpha \gamma \aa F &= 3yz^2, & \gamma^2 \aa F &= 6xyz.
\end{align*}
Observe that the nonzero derivatives listed above are linearly independent: in fact no
monomial appears in more than one of them.
So $(F^\perp)_2$ is spanned by $\alpha^2$.
This shows that the Hilbert function of $A^F$ is $1,3,5,5,3,1$.
In particular $\al(F) = 18$.

Observe also that $\alpha \aa F= y z^3$, $\alpha^2 \aa F = 0$.
By Theorem~\ref{thm: waring bound}, $r(F) \geq \dim \Diff(y z^3) = 8$.
If $r(F) = 8$ then $r(F) \geq \dim \Diff(F) - \dim \Diff(y z^3) = 18 - 8 = 10$,
by Corollary~\ref{corollary: improve by 1}.
So $r(F) > 8$.

Now we rule out the possibility $r(F) = 9$.
Suppose to the contrary $F = \ell_1^5 + \dotsb + \ell_9^5$.
Proposition~\ref{prop: bound points off alpha} shows at least $8$ of the $[\ell_i]$ lie off of the hyperplane $\PP V(\alpha)$;
but if all $9$ lie off of the hyperplane, then by Proposition~\ref{prop: no point bound}
$r(F) \geq \dim \Diff(F) - \dim \Diff(\alpha \aa F) = 10$.
So say $\ell_1,\dotsc,\ell_8$ lie off of $V(\alpha)$ and $\ell_9 = ay+bz$ lies on $V(\alpha)$.
Let $G = F - (ay+bz)^5 = \ell_1^5 + \dotsb + \ell_8^5$, so that $r(G) = 8$.
Note $\alpha \aa  G = \alpha \aa F = y z^3$.
We compute again:
\begin{align*}
  \alpha^2 \aa G &= 0, & \beta^2 \aa G &= 12 y^2 z - 20a^2(ay+bz)^3, \\
  \alpha \beta \aa G &= z^3, & \beta \gamma \aa G &= 3xz^2 + 4y^3 - 20ab(ay+bz)^3, \\
  \alpha \gamma \aa G &= 3yz^2, & \gamma^2 \aa G &= 6xyz - 20b^2(ay+bz)^3.
\end{align*}
If $a \neq 0$ then $\alpha \beta \aa G$, $\alpha \gamma \aa G$, $\beta^2 \aa G$
  are linearly independent
as $\beta^2 \aa G$ is the only one with a nonzero $y^3$ term.
If $a = 0$ then the same three derivatives are still linearly independent as they are distinct monomials.
And $\beta \gamma \aa  G$, $\gamma^2 \aa G$ are linearly independent modulo the other derivatives because they involve
different monomials with $x$.
In conclusion, the nonzero derivatives of $G$ listed above are linearly independent, so $(G^\perp)_2$ is spanned by $\alpha^2$.
It follows that $A^G$ has Hilbert function $1, 3, 5, 5, 3, 1$, the same as $A^F$.

Now the same argument applies to $G$: $\alpha \aa G = \alpha \aa F = y z^3$, $\alpha^2 \aa G = 0$,
so $r(G) \geq \dim \Diff(yz^3) = 8$, and if $r(G) = 8$ then $r(G) \geq \dim \Diff(G) - \dim \Diff(\alpha \aa G) = 10$,
hence $r(G) > 8$, by Corollary~\ref{corollary: improve by 1}.
This contradicts the construction of $G$ which shows $r(G) = 8$.

It follows that $r(F) > 9$, so $r(F) = 10$.
\end{proof}

\begin{remark}
The result of Theorem~\ref{thm: n=3 greater than monomial} is the best possible for degrees $d = 3,5$:
the result $\rmax(3,d) \geq 1+((d+1)/2)^2$ is equality for these degrees.
For other degrees, and for $n > 3$, one may ask if this bound can be improved.
Two potential routes for improvement suggest themselves.
First, Carlini, et al, show a more general and potentially stronger version of Theorem~\ref{thm: waring bound},
see \cite[Corollary~3.4]{Carlini:2015sp}.
Second, one might try modifying the proof of Theorem~\ref{thm: general case higher rank}
by taking $G$ of apolar length $2$ less than the general apolar length,
and showing that in appropriate cases $\PP(D^{-1}(T_{k-1} \aa G))$ is disjoint from not only a general translate of the
Veronese but in fact from a general translate of the secant variety of the Veronese.
\end{remark}

\bigskip

\renewcommand{\MR}[1]{{}}

\bibliographystyle{amsplain}


\providecommand{\bysame}{\leavevmode\hbox to3em{\hrulefill}\thinspace}
\providecommand{\MR}{\relax\ifhmode\unskip\space\fi MR }
\providecommand{\MRhref}[2]{%
  \href{http://www.ams.org/mathscinet-getitem?mr=#1}{#2}
}
\providecommand{\href}[2]{#2}

\bigskip

\end{document}